\documentclass[a4paper,11pt]{article}
\usepackage[pagewise]{lineno}
\usepackage{amsmath}
\usepackage{amssymb}
\usepackage{mathrsfs}
\usepackage{amsfonts}
\usepackage{amsthm}
\usepackage{pstricks, pst-node, pst-text, pst-3d,psfrag}
\usepackage{graphicx}
\usepackage{indentfirst}
\usepackage{enumerate}
\usepackage{subfigure}
\usepackage{cite}
\usepackage[colorlinks=true]{hyperref}
\usepackage{authblk}
\hypersetup{urlcolor=blue, citecolor=red}

\theoremstyle{plain}
\newtheorem{thm}{Theorem}[section]
\newtheorem{prop}[thm]{Proposition}

\newtheorem{lem}[thm]{Lemma}

\theoremstyle{definition}

\newtheorem{rmk}[thm]{Remark}

\numberwithin{equation}{section}

\makeatletter 
\@addtoreset{equation}{section}
\makeatother  

\begin{document}
\title{Prevalent Behavior of Smooth Strongly Monotone Discrete-Time Dynamical Systems\thanks{Supported by NSF of China No.11825106, 11771414 and 11971232, Wu Wen-Tsun Key Laboratory of Mathematics, University of Science and Technology of China.}}
\author{Yi Wang}
\author{Jinxiang Yao}
\author{Yufeng Zhang\thanks{Corresponding author: zyfp@mail.ustc.edu.cn (Y. Zhang).}}

\affil{School of Mathematical Sciences\\ University of Science and Technology of China\\ Hefei, Anhui, 230026, P. R. China}
\renewcommand\Authands{ and }
\date{}

  \maketitle
\begin{abstract}
For $C^{1}$-smooth strongly monotone discrete-time dynamical systems, it is shown that ``convergence to linearly stable cycles'' is a prevalent asymptotic behavior in the measure-theoretic sense. The results are then applied to classes of time-periodic parabolic equations and give new results on prevalence of convergence to periodic solutions. In particular, for equations with Neumann boundary conditions on convex domains, we show the prevalence of the set of initial conditions corresponding to the solutions that converge to spatially-homogeneous periodic solutions. While, for equations on radially symmetric domains, we obtain the prevalence of the set of initial values corresponding to solutions that are asymptotic to radially symmetric periodic solutions.\par
\end{abstract}
\noindent\section{Introduction}
Prevalence is a frequently used notion, first introduced by Christensen \cite{C72} and Hunt, Sauer and Yorke \cite{HSY93}, that describes properties of interest occurs for ``almost surely" in an infinite-dimensional space from a probabilistic or measure-theoretic perspective. It is a natural generalization to separable
Banach spaces of the notion of Lebesgue measure zero for Euclidean spaces (see definition in Section \ref{sec2}). In particular, on $\mathbb{R}^n$, it is equivalent to the notion of ``Lebesgue almost everywhere". Over decades since its development, prevalence have undergone extensive investigations. We refer to \cite{HSY93,K97,BV10,EN20} (and references therein) for more details.

In dynamical systems, prevalence is parallel to another more classical notion called genericity, which formulates in the topological sense the properties of interest occurs residually (i.e., on a countable intersection of open dense subsets in a Baire space). Many examples in dynamical systems are known to be both generic in the topological sense and prevalent in measure-theoretic sense; that is, these two notions of typicality could coincide with each other in many circumstances (see, e.g. \cite{BV10,K97}). However, on the other hand, there are also many cases for which prevalence differs from topological genericity. As a matter of fact, in many cases properties that are generic in the topological sense are not prevalent in the measure-theoretic sense, and vise-versa (see, e.g. \cite{BV10,K97}).

Monotone dynamical systems are abundant and important sources of topological genericity and prevalence. The theory of monotone dynamical systems grew out of the series of ground-breaking work of M. W. Hirsch \cite{H84,H88,H82,H85} and H. Matano \cite{M79,M86}; largely focusing on ordinary differential equations and parabolic partial differential equations. Since then, the theory and applications have been extended to discrete-time dynamical systems, non-autonomous systems, as well as random/stochastic systems. One may refer to \cite{I02,FWW17,MH05,JX05,P,WY98,H,H17} for the earlier and recent developments.

The core to the huge success of developing the theory and applications of a strongly  monotone semiflow is the so called {\it Hirsch's generic convergence theorem} \cite{H85}, concluding that the generic precompact orbit approaches the set of equilibria (also referred as generic quasi-convergence).
Later on, for $C^1$-smooth strongly monotone semiflows, the improved {\it generic convergence} was obtained by Pol\'{a}\v{c}ik \cite{P89,P90} and Smith and Thieme \cite{ST91}.  Very recently, generic Poincar\'{e}-Bendixson Theorem was established by Feng, Wang and Wu \cite{FWW19,FWW17} for smooth flows strongly monotone with respect to cone of rank-$2$.

For strongly monotone discrete-time systems (mappings), however, there is no result analogous
to Hirsch's generic convergence theorem is available if they are merely continuous (or even Lipschitz).
Whether any characterization of typical dynamics is possible in this case or not is
still an open problem, unless certain smoothness assumption is imposed. Pol\'{a}\v{c}ik and Tere\v{s}\v{c}\'{a}k \cite{PI91} first proved that the
{\it generic convergence to cycles} occurs provided that the mapping $F$ is of class $C^{1,\alpha}$ (i.e., $F$ is a $C^1$-map with a locally $\alpha$-H\"{o}lder derivative $DF$, $\alpha\in (0,1]$). Here, a cycle means a periodic orbit of $F$. For the lower regularity of $F$, Tere\v{s}\v{c}\'{a}k \cite{T94} and Wang and Yao \cite{WY20} utilized different approaches to obtain the generic convergence to cycles for $C^1$-smooth strongly monotone discrete-time systems.

A drawback of topological genericity is that open dense subsets can be arbitrarily small in terms of measure. Hirsch \cite{H88} further showed that his generic convergence theorem also holds in a measure-theoretical sense of Gaussian measure. However, since it is based on the normal
distributions in statistics, Gaussian measure is not translation-invariant.

Hence, a natural interesting question arises whether the prevalent dynamics analogue hold for strongly monotone systems. To the best of our knowledge,
Enciso, Hirsch and Smith \cite{GMH08} first tackled the problem and investigated the prevalent behavior of a strongly monotone semiflow $\Phi_t$. Among others, they \cite[Theorem 1]{GMH08} proved that the set of points that converge to a linearly stable equilibrium is prevalent. Here, an equilibrium is linearly stable if the spectral radius of the Frechet derivative $D\Phi_t$ at the equilibrium is no more than $1$ for all $t>0$.

In the present paper, we shall focus on prevalent behavior of the strongly monotone discrete-time dynamical systems. To be more specific, we will first show that, for $C^1$-smooth strongly monotone discrete-time system, the set of points converge to a linearly stable cycle is prevalent (see Theorem A). Here, a cycle is linearly stable if the spectral radius of the derivative $DF^p$ along the cycle (of period $p$) is no more than 1 (see Section \ref{sec2}).

It deserves to point out that there are a few major differences
of research approaches between discrete and continuous time strongly monotone systems. For instance, Limit Set Dichotomy (see, e.g. \cite[Theorem 4, p.121]{GMH08} or \cite{MH05,H}), which is one of the fundamental building blocks for continuous-time systems, are no longer valid for discrete-time systems (see, e.g. \cite{P,DP91,P92}). Thus, the situation is quite different with discrete-time systems as one has no a priori information on the structure of limit sets of typical trajectories.
Consequently, novel techniques are needed to understand the prevalent dynamics. Our approach is based on a critical insight for the inherent structure of discrete-time strongly monotone systems, called \emph{Dynamics Alternatives} (see Lemma \ref{dic1}),  which was first discovered by Pol\'{a}\v{c}ik and Tere\v{s}\v{c}\'{a}k \cite{PI91} for $C^{1,\alpha}$-mappings and was recently improved by the two of present authors \cite{WY20} for $C^1$-mappings.
Together with a useful tool of upper (respectively, lower) $\omega$-unstable sets introduced by Tak\'{a}\v{c} \cite{P92}, we utilize the $C^1$-Dynamics Alternatives to accomplish our approach.

Moreover, motivated by our approach, we show the prevalence of asymptotic symmetry in
discrete-time strongly monotone systems on which a compact connected
group $G$ acts. This gives the prevalent analogues to the works by Mierczy\'{n}ski and Pol\'{a}\v{c}ik \cite{M89} and Tak\'{a}\v{c} \cite{P920} on the topological genericity of asymptotic symmetry in
continuous and discrete strongly monotone dynamical systems admitting group actions (see also \cite{W09}).
For these results, an interesting feature of is that no
smoothness assumption was imposed on the systems. Combining with Theorem A, we further obtain that, if in addition the discrete-time system is $C^1$-smooth, then the set of points that are asymptotic to $G$-symmetric cycles is prevalent (see Theorem B).

Our main results will be applied to obtain the dynamics of prevalent initial values for nonlinear time $\tau$-periodic parabolic equations with proper boundary value problems. By applying Theorem A, we first obtain the prevalence of the set of initial points corresponding to the solutions that are asymptotic to a linearly stable subharmonic periodic solution (i.e., solution whose minimal period is $k\tau$ ($k\in \mathbb{N}$), a nontrivial multiple of the period of the equations). In particular, for the periodic parabolic equations with Neumann boundary conditions on convex domains, we proved the prevalence of the set of initial conditions corresponding to the solutions that converge to spatially homogeneous $\tau$-periodic solutions. We further apply Theorem B to obtain asymptotic symmetry of prevalent solutions for time-period parabolic equations on a radially symmetric domain. In particular, due to the close interplay between spatial geometry and the temporal evolution of solutions, we obtained the prevalence of the set of initial values corresponding to solutions that are asymptotic to radially symmetric $\tau$-periodic solutions.\par

This paper is organized as follow. In Section \ref{sec2}, we agree on some notations, give relevant definitions and state our main results. We give the proofs of our main theorems in Section \ref{sec3}. Section \ref{sec5} is devoted to the study the prevalent dynamical behavior of the solutions for the nonlinear time-periodic parabolic equations.\par


\noindent\section{Notations and Main Theorems}\label{sec2}
In this section, we will fix some notations, make some preliminaries and formulate our main theorems at the end of this section. \par
Let $(\mathbb{B},\|\cdot\|)$ be a Banach space with norm $\|\cdot\|$. A closed convex subset $K\subset\mathbb{B}$ is called a \emph{closed convex cone} if
(i) if $x\in K$ and $\alpha>0$, then $\alpha x\in K$;
(ii) if $x,y\in K$, then $x+y\in K$;
(iii) $K\cap (-K)=\{0\}$.
For $x,y\in \mathbb{B}$, we write $x\leq y$ iff $y-x\in K$; $x<y$ iff $x\leq y$ and $x\neq y$; $x\ll y$ iff $y-x\in$ $\rm{Int}$$K$.
$\mathbb{B}$ is called an \emph{ordered Banach space} if it is endowed with an order relation, induced by a cone $K$. In particular, $\mathbb{B}$ is called a \emph{strongly ordered Banach space} if the interior $\rm{Int}$$K$ of $K$ is nonempty. For $a,b\in \mathbb{B}$ with $a\leq b$, we denote the \emph{closed order interval} by $[a,b]=\{x\in \mathbb{B}:a\leq x\leq b\}$ and the \emph{open order interval} by $[[a,b]]=\{x\in \mathbb{B}:a\ll x\ll b\}$. In particular, we write $[a,\infty]]=\{x\in \mathbb{B}:x\geq a\}$ and $[[\infty,b]=\{x\in \mathbb{B}:x\leq b\}$. Let $X\subset\mathbb{B}$, we denote by $\widehat{X}$ the set $X$ with the order-topology generated from open order interval $[[a,b]]$. A subset $X\subset\mathbb{B}$ is called \emph{order-convex} in $\mathbb{B}$ if $[a,b]\subset X$ whenever $a,b\in X$ and $a<b$. 

Let $X\subset\mathbb{B}$ and $F:X\rightarrow X$ is a continuous map. The map $F$ is called \emph{monotone} if $x\leq y$ implies $Fx\leq Fy$, and is called \emph{strongly monotone} if $Fx\ll Fy$ whenever $x<y$.\par

A \emph{semi-orbit} of $x\in X$ is $\mathcal{O}^{+}(x):=\{F^{n}x\}_{n\geq 0}$, and the \emph{$\omega$-limit set} of $x$ is $\omega(x):=\bigcap_{k\geq0}\overline{\mathcal{O}^{+}(F^{k}x)}$. We say that the map $F$ is \emph{$\omega$-compact} in a subset $B\subset X$ if $\mathcal{O}^{+}(x)$ has compact closure in $X$ for each $x\in B$ and $\cup_{x\in B}\omega(x)$ has compact closure in $X$. A point $x\in X$ is called \emph{$p$-periodic} for some integer $p\geq 1$, if $F^{p}(x)=x$ and $F^{l}x\neq x$ for $l=1$, $2$, $\cdots$, $p-1$. A set $B$ is called a \emph{cycle} if $B=\mathcal{O}^{+}(x)$ for some periodic points $x$. A cycle is \emph{linearly stable} if the spectral radius of the derivative $DF^{p}$ along the cycle (of period $p$) is no more than $1$. A set $B\subset X$ is \emph{unordered} if it does not contain points $x,y$ such that $x<y$.  Given any $x\in X$, we define the \emph{upper} and  \emph{lower $\omega$-limit sets} of $x$ in $X$ by
$$\omega_{+}(x):=\bigcap_{\substack{u\in X\\u\gg x}}\overline{(\bigcup_{\substack{y\in X\\ u\geq y>x}}\omega(y))} \quad and\quad
\omega_{-}(x):=\bigcap_{\substack{u\in X\\u\ll x}}\overline{(\bigcup_{\substack{y\in X\\ u\leq y<x}}\omega(y))},$$
respectively. Clearly, $\omega_{+}(x)$ (resp. $\omega_{-}$) is non-empty if $F$ is $\omega$-compact in every closed order interval $[a,b]$ in $X$. Hereafter, we write 
$$\mathcal{U}_{+}=\{x\in X:\omega_{+}(x)\neq \omega(x)\} \quad and\quad \mathcal{U}_{-}=\{x\in X:\omega_{-}(x)\neq \omega(x)\}$$
as the \emph{upper $\omega$-unstable set} and \emph{lower $\omega$-unstable set} in $X$, respectively.\\\par

Throughout the paper, we will assume the following standing assumptions:
\begin{description}
  \item[(H1)] $\mathbb{B}$ is a separable strongly ordered Banach space with a closed-convex cone $K$; and $F:\mathbb{B}\rightarrow \mathbb{B}$ is a strongly monotone map.\par
  \item[(H2)] Let $X\subset\mathbb{B}$ be an $F$-invariant order-convex open subset; and $F$ is $\omega$-compact in every closed order interval $[a,b]$ in $X$.\par
  \item[(H3)] $F:\mathbb{B}\rightarrow \mathbb{B}$ is a $C^{1}$-map, such that for any $x\in \mathbb{B}$ the derivative $DF(x)$ is a compact strongly positive operator, i.e., $DF(x)v\gg 0$ whenever $v>0$.
\end{description}\par

In the following, we give the definition of \emph{prevalence} (see, e.g. \cite{BV10,HSY93,GMH08}). A Borel set $W\subset \mathbb{B}$ is \emph{shy} if there exists a nonzero compactly supported Borel measure $\mu$ on $\mathbb{B}$ such that $\mu(x+W)=0$ for every $x\in\mathbb{B}$. More generally, a set is called \emph{shy} if it is contained in a shy Borel set. A set is \emph{prevalent} if its complement is shy. Given $A\subset \mathbb{B}$, we say that a set $W$ is \emph{prevalent in} $A$ if $A\setminus W$ is shy.\par
In general situation, shy sets have the following properties: (i) Every subset of a shy set is shy; (ii) Every translation of a shy set is shy; (iii) No nonempty open set is shy. (iv) Every countable union of shy sets is shy (\cite{HSY93,BV10}). In finite-dimensional spaces, a set $M$ is shy if and only if it has Lebesgue measure zero (see, e.g. \cite[Proposition 2.5]{BV10} or \cite[Fact 6]{HSY93}).\par


Denote by $\mathfrak{C}_{P}$ the set of states $x\in X$ such that $\omega(x)$ is a linearly stable cycle, i.e.,
$$\mathfrak{C}_{P}:=\{x\in X:\omega(x)\ is\ a\ linearly\ stable\ cycle\}.$$

\vspace{2 ex}

Our first main theorem is as follows:\par

\vspace{2 ex}

\noindent$\textbf{Theorem A}.$
\emph{Assume that \emph{(H1)}, \emph{(H2)} and \emph{(H3)} hold. Then $\mathfrak{C}_{P}$ is prevalent in $X$.}\par

\vspace{2 ex}

In order to state our next main theorem, we need to introduce an additional assumption. Let $G$ be a compact metrizable topological group. A mapping $\Gamma:G\times X\rightarrow X$ is called a \emph{group action} of $G$ on $X$ if it is jointly continuous and $g\mapsto\Gamma(g)\equiv\gamma(g,\cdot)$ is a group homeomorphism of $G$ into Hom($X$), the group of all homeomorphisms of $X$ onto itself. We say that $\gamma$ is \emph{increasing} if, for each $g\in G$, the mapping $\Gamma(g):X\rightarrow X$ is increasing, i.e., $x\leq y$ in $X$ implies $\Gamma(g)x\leq\Gamma(g)y$. We say that $\gamma$ \emph{commutes} with a map $F$ if, for each $g\in G$, the mapping $\Gamma(g)$ commutes with $F$. For briefly, We write $\gamma(g,x)\equiv g\cdot x$. A point $x\in X$ is called \emph{G-symmetric} if $g\cdot x=x$ for all $g\in G$. A subset $S\subset X$ is called \emph{G-symmetric} if all points in $S$ are $G$-symmetric.\par

\begin{description}
  \item[(H4)] $G$ is a compact connected metrizable topological group acting on $X$ in such a way that its action is increasing and commutes with $F$.
\end{description}

\noindent$\textbf{Theorem B}.$
\emph{Assume that \emph{(H1)}, \emph{(H2)} and \emph{(H4)} hold. Then the subset $$S=\{x\in X:\omega(x)\ is\ G\text{-}symmetric\}$$ is prevalent in $X$. Moreover, if in addition \emph{(H3)} holds, then the set of initial points corresponding to semiorbits that are asymptotic to $G$-symmetric linearly stable cycles, is prevalent in $X$.}

\vspace{2 ex}

\begin{rmk}
Under assumptions (H1)-(H3), it has been proved in Tere\v{s}\v{c}\'{a}k \cite{T94} and Wang and Yao \cite{WY20} that $\mathfrak{C}_{P}$ contains an open and dense subset of $X$ (hence, $\mathfrak{C}_{P}$ is generic in $X$). Our Theorem A shows that $\mathfrak{C}_{P}$ is also prevalent in $X$. This entails the fact that \emph{``convergence to linearly stable cycles''} in $C^{1}$-smooth discrete-time strongly monotone systems are both prevalent in measure-theoretic sense and  generic in the topological sense.
\end{rmk}

\begin{rmk}
In \cite{M94}, Mierczy\'nski smartly constructed an example of a strongly monotone discrete-time system, in which the subset $\mathcal{M}\cup \mathcal{N}$ is open and dense (hence, generic), but not prevalent in $X$ (see \cite[p.1492]{M94}). Here $\mathcal{M}$ is the set of points whose orbits are eventually monotone, and $\mathcal{N}$ is the set of those convergent $\xi$ for which there is a simply ordered arc $L\ni\xi$, $L\subset X\backslash \mathcal{M}$ (see more detail of $\mathcal{M}$ and $\mathcal{N}$ in \cite[p.1477]{M94}). We here point out that $\mathcal{M}\cup \mathcal{N}$ is actually contained in $\mathfrak{C}_{P}$, which means that $\mathcal{M}\cup \mathcal{N}$ is not large enough to guarantee its prevalence in $X$.
\end{rmk}

\begin{rmk}
Theorem B provides the prevalent analogues to the works by Mierczy\'nski and Pol\'a\v{c}ik \cite{M89} and Tak\'a\v{c} \cite{P920} on the topological genericity of symmetry in continuous and discrete strongly monotone dynamical systems admitting group actions.
\end{rmk}


\noindent\section{Proof of the main Theorems}\label{sec3}
In this section, we will prove our main Theorems. 
For this purpose, we need several useful propositions, which describe the structures of the upper (resp. low) $\omega$-unstable sets $\mathcal{U}_{+}$ (resp. $\mathcal{U}_{-}$), as well as a useful sufficient condition to guarantee a set to be shy in $X$.\par



First, we focus on the structure of $\mathcal{U}_{+}$ and $\mathcal{U}_{-}$. Some additional notations are introduced. A pair $(A,B)$ of subsets $A,B$ of $X$ is called an \emph{order decomposition} of $X$ if it satisfies: (i) $A\neq\emptyset$ and $B\neq\emptyset$; (ii) $A$ and $B$ are closed; (iii) $A$ is lower closed (i.e., $[[-\infty,a]\subset A$ whenever $a\in A$) and $B$ is upper closed (i.e., $[b,\infty]]\subset B$ whenever $b\in B$); (iv) $A\cup B=X$; (v) Int$(A\cap B)=\emptyset $.\par

An order decomposition $(A,B)$ of $X$ is called \emph{invariant} if $F(A)\subset A$ and $F(B)\subset B$. The set $H=A\cap B$ (possibly empty) is called the \emph{boundary} of the order decomposition $(A,B)$ of $X$. A \emph{$d$-hypersurface} is a nonempty subset $H$ of $X$ such that $H=A\cap B$ for some order decomposition $(A,B)$ of $X$. The boundary $H$ of the order decomposition $(A,B)$ of $X$ satisfies $H=\partial A=\partial B$, where ``$\partial$" is the boundary symbol in $X$, and $H$ is invariant whenever $(A,B)$ is invariant. Clearly, a $d$-hypersurface $H$ never contains two strongly ordered points ($x\ll y$). In particular, every nonempty, unordered, invariant set $G\subset X$ is contained in some invariant $d$-hypersurface $H\subset X$ (see, e.g. \cite[Proposition 1.2]{P92}).\par
A system $\Gamma$ of invariant order decompositions of $X$ is called an \emph{invariant order resolution} of $X$ if $\Gamma$ satisfies (i) If $(A_{1}B_{1})$ and $(A_{2},B_{2})\in \Gamma$, then either $A_{1}\subset A_{2}$ or $A_{2}\subset A_{1}$; (ii) If $\mathcal{O}^{+}(x)$ is unordered, then $O^{+}(x)\subset H=A\cap B$ for some $(A,B)\in \Gamma$. The existence of the invariant order resolution can be guaranteed by \cite[Theorem 4.2]{P92}.\par
Now we give the following structures of the $\mathcal{U}_{+}$ and $\mathcal{U}_{-}$, which are critical for the proof of the prevalence. \par
\begin{lem}\label{borel}
Assume that \emph{(H1)}, \emph{(H2)} hold. Let $\Gamma$ be the invariant order resolution of $X$. Then $\mathcal{U}_{+}$ is the union of at most countably many sets $M_{i}$; each $M_{i}$ are relatively open in $H_{i}$, where $H_{i}=A\cap B$ for some $(A,B)\in \Gamma$. A corresponding result holds for $x\in\mathcal{U}_{-}$.
\end{lem}
\begin{proof}
See Tak\'a\v{c} \cite[Corollary 5.6]{P92}.
\end{proof}
\begin{rmk}
If, in addition, $X$ is order-open ($X$ is open in the order-topology of $\widehat{\mathbb{B}}$), then each $M_{i}$ is a Lipschitz manifold of codimension one in $\widehat{\mathbb{B}}$ (see the detail in Tak\'a\v{c} \cite[Corollary 5.6]{P92}).
\end{rmk}


\begin{prop}\label{exp}
Assume that \emph{(H1)}, \emph{(H2)} hold. Then\par
\emph{(i).} If $x\in \mathcal{U}_{+}$, then there exists some $a_{x}\gg x$ such that $\omega(y)=\omega_{+}(x)$ for any $y\in [[x,a_{x}]]$; and hence, $[[x,a_{x}]]\cap\mathcal{U}_{+}=\emptyset$.\par
\emph{(ii).} Let $x,y\in \mathcal{U}_{+}$ and $x\ll y$. Then the two corresponding open sets obtained in \emph{(i)} are mutually disjoint, i.e., $[[x,a_{x}]]\cap[[y,a_{y}]]=\emptyset$.\par
A corresponding result holds for $x,y\in\mathcal{U}_{-}$.


\end{prop}
\begin{proof}
(i). See Tak\'a\v{c} \cite[Proposition 3.3]{P92}. (ii). Let $a_{x}$, $a_{y}$ be obtained from (i) for $x$ and $y$, respectively. Suppose that $b\in [[x,a_{x}]]\cap[[y,a_{y}]]\neq\emptyset$. Then $x\ll y\ll b \ll a_{x}$; hence, $y\in [[x,a_{x}]]$, which implies that $y\in [[x,a_{x}]]\cap \mathcal{U}_{+}\neq\emptyset$, contradicting (i) of Proposition \ref{exp}.
\end{proof}


\begin{prop}\label{imp}
Assume that \emph{(H1)}, \emph{(H2)} hold. Then both $\mathcal{U}_{+}$ and $\mathcal{U}_{-}$ are Borel sets.
\end{prop}
\begin{proof}
We only prove that $\mathcal{U}_{+}$ is a Borel set, since the proof of $\mathcal{U}_{-}$ is analogous. Let $\Gamma$ be the invariant order resolution of $X$. Then, by Lemma \ref{borel}, $\mathcal{U}_{+}$ is the union of at most countably many subsets $M_{i}$. Each $M_{i}$ is relatively open with respect to $H_{i}$, with $H_{i}=A_{i}\cap B_{i}$ for certain $(A_{i},B_{i})\in \Gamma$. In other words,
$$\mathcal{U}_{+}=\bigcup_{i\in \Lambda}M_{i}=\bigcup_{i\in \Lambda}(H_{i}\cap V_{i})=\bigcup_{i\in \Lambda}((A_{i}\cap B_{i})\cap V_{i}),$$
where $\Lambda$ is a countable index set; and $V_{i}$ ($i\in \Lambda$) are open subsets in $X$. Noticing that $A_{i}$ and $B_{i}$ are closed due to the definition of order decomposition, we obtain that $\mathcal{U}_{+}$ is Borel.
\end{proof}

Next, we introduce a sufficient condition which guarantees a Borel subset of $X$ to be shy.
\begin{lem}\label{shy}
Let $W\subset X$ be a Borel subset and assume that there exists a vector $v\gg0$ such that $L\cap W$ is countable for every straight line $L$ parallel to $v$. Then $W$ must be shy in $X$.
\end{lem}
\begin{proof}
See Enciso, Hirsch and Smith \cite[Lemma 1]{GMH08}.
\end{proof}

Before we give the proof, we mention the following critical insight for the inherent structure of the system.
\begin{lem}[Dynamics alternatives]\label{dic1}
Assume that \emph{(H1)}-\emph{(H3)} hold. For any $x\in X$, one of the following situations must occurs: \\
\emph{(Alta).} $x\in \mathfrak{C}_{P}$;\\
\emph{(Altb).} there exists some $\delta>0$ such that for any $y\in X$ satisfying $y<x$ or $y>x$, one has $$\limsup_{n\rightarrow\infty}\|F^{n}(y)-F^{n}(x)\|\geq\delta.$$
\end{lem}
\begin{proof}
See Wang and Yao \cite[Theorem 2.1]{WY20} (or Pol\'a\v{c}ik and Tere\v{s}\v{c}\'{a}k \cite[Theorem 4.1]{PI91} for $C^{1,\alpha}$-version).
\end{proof}

Now we are ready to prove the our main Theorems.\\\par

\noindent$\textbf{Proof of Theorem A}.$
Let $M=X\backslash\mathfrak{C}_{P}$ be the complement of $\mathfrak{C}_{P}$ in $X$.
We first show that $M\subset\mathcal{U}\triangleq\mathcal{U}_{+}\cap\mathcal{U}_{-}$. For this purpose, we utilize Dynamics Alternatives in Lemma \ref{dic1}, which state that, for each $x\in X$, either (Alta) $x\in \mathfrak{C}_{P}$; or (Altb) there exists some $\delta>0$ such that for any $y\in X$ satisfying $y>x$ or $y<x$, one has $\limsup_{n\rightarrow\infty}\|F^{n}y-F^{n}x\|\geq\delta$.\par
By virtue of Lemma \ref{dic1}, any point $x\in M$ can only satisfy (Altb). Then, together with the definition of upper and lower $\omega$-limit sets, we obtain that $\omega(x)\neq\omega_{+}(x)$ and $\omega(x)\neq\omega_{-}(x)$, which entails that $x\in \mathcal{U}_{+}\cap\mathcal{U}_{-}\triangleq\mathcal{U}$.\par
Next, we will prove that both $\mathcal{U}_{+}$ and $\mathcal{U}_{-}$ are shy. Again, we just show that $\mathcal{U}_{+}$ is shy. The case for $\mathcal{U}_{-}$ is similar. By virtue of Lemma \ref{shy}, it suffices to show that
 $\mathcal{U}_{+}\cap L_{v}$ is countable, where $L_{v}$ is any straight line parallel to any fixed vector $v\gg0$. To this end, fix any positive vector $v\gg0$. Recalling that $\mathcal{U}_{+}$ is Borel (by Proposition \ref{imp}),
we can define a set-function $R$ from $\mathcal{U}_{+}\cap L_{v}$ to $\mathcal{B}(L_{v})$ as $$R:\mathcal{U}_{+}\cap L_{v}\rightarrow \mathcal{B}(L_{v});\  x\mapsto f(x)\triangleq L_{v}\cap[[x,a_{x}]],$$
where $\mathcal{B}(L_{v})$ are the set of all Borel subsets of $L_{v}$, and $[[x,a_{x}]]$ is from Proposition \ref{exp}(i). By Proposition \ref{exp}(ii), it is clear that $R$ is injective; and moreover, for each $x\in \mathcal{U}_{+}\cap L_{v}$, $[[x,a_{x}]]\cap L_{v}$ is mutually disjoint open in $L_{v}$. Consequently, the image of $R$ is an at most countable subset in $\mathcal{B}(L_{v})$. Since $R$ is injective, it follows that $\mathcal{U}_{+}\cap L_{v}$ is at most countable. Thus, by Lemma \ref{shy}, we have proved that $\mathcal{U}_{+}$ is shy.\par
Finally, noticing that $M\subset\mathcal{U}\triangleq\mathcal{U}_{+}\cap\mathcal{U}_{-}$, we obtain that $M$ is shy, which implies that $\mathfrak{C}_{P}$ is prevalent in $X$. We have completed the proof.
\hfill$\square$

\vspace{4 ex}

\noindent$\textbf{Proof of Theorem B}.$
Let $M=X\backslash S$. Then one has $M\subset \mathcal{U}$ (see Tak\'a\v{c}\cite[Theorem 0.3]{P920}). By the same arguments in the proof of Theorem A, we get that $\mathcal{U}_{+}$ and $\mathcal{U}_{-}$ are shy sets; and hence, $M$ is shy in $X$, which entails that $S$ is prevalent in $X$.\par
If in addition (H3) holds, then $\mathfrak{C}_{P}$ is prevalent in $X$ by Theorem A. Recall that $S$ is prevalent in $X$. Then, we have $\mathfrak{C}_{P}\cap S$ is prevalent in $X$.
\hfill$\square$

\noindent\section{Application to parabolic equations}\label{sec5}
In this section, we will apply our main theorems to investigate the prevalent dynamics of parabolic equations with proper boundary value problems.\par
\vspace{2 ex}

\noindent\textbf{Example 1 (Prevalent dynamics of parabolic equations on bounded domains)}
Consider the following initial-boundary value problem for parabolic equations
\begin{eqnarray}\label{E1}
        \frac{\partial u}{\partial t} -\triangle u &=& f(t,x,u,\nabla u), \quad\quad\quad\ x\in \Omega,\ t>0,\notag\\
       \mathcal{B}u &=& 0, \quad\quad\quad\quad\quad\quad\quad\quad\ x\in \partial\Omega,\ t>0,\\
       u(0,x) &=& u_{0}(x).\notag
\end{eqnarray}
Here $\Omega\subset\mathbb{R}^{m}$ is a bounded domain with boundary $\partial\Omega$ of class $C^{2+\theta}$ for some $\theta\in (0,1)$. Of course, $\Delta$ is the Laplacian.
The nonlinearity $f:(t,x,u,\xi)\mapsto f(t,x,u,\xi)$ are assumed to satisfy:\par

\vspace{2 ex}

(\textbf{N1}) $f$ is continuous and locally H\"older continuous in $(t,x)$ uniformly for $(u,\xi)$ in bounded subset of $\mathbb{R}\times\mathbb{R}^{m}$, and $\frac{\partial f}{\partial u}$ and $\frac{\partial f}{\partial\xi_{i}}(i=1,\ldots,m)$ exist and are continuous with respect to $(u,\xi)$.\par
(\textbf{N2}) $f$ is periodic in $t$ of given period $\tau>0$.\par

\vspace{2 ex}

We consider a time-independent regular linear boundary operator $\mathcal{B}$ on $\partial\Omega$ of \emph{Dirichlet} ($\mathcal{B}u=u$) or \emph{Neumann} ($\mathcal{B}u=\frac{\partial u}{\partial v}$) type. Here $v$ is the unit outward normal vector field on $\partial\Omega$.\par
We also assume that\par

\vspace{2 ex}

(\textbf{N3}) There exists an $\epsilon>0$ and a continuous function $c:\mathbb{R}^{+}\rightarrow\mathbb{R}^{+}$ such that
$$|f(t,x,u,\xi)| \leq c(p)(1+|\xi|^{2-\epsilon})\quad
(p\geq0,\ (t,x,u,\xi)\in [0,\tau]\times \bar{\Omega} \times [-p,p]\times\mathbb{R}^{m}).$$
\indent(\textbf{N4}) There is some $\kappa>0$ such that
$uf(t,x,u,0)<0\quad ((t,x)\in [0,\tau]\times\bar{\Omega},|u|\geq\kappa).$

\vspace{2 ex}

Let $A$ be the realization of $\triangle$ under boundary condition, that is, $A$ is the operator $u\mapsto \triangle u$ with domain $$D(A)=X^{1}\triangleq\{u\in W^{2,p}(\Omega): \mathcal{B}u = 0\}.$$
Let $X^{\alpha}$, $0\leq\alpha<1$, be the fractional power spaces associated with $A$. We choose $\frac{1}{2}+\frac{m}{2p}<\alpha<1$ such that $X^{\alpha}\subset C^{1+\gamma}(\bar{\Omega},\mathbb{R})$ with continuous inclusion for $\gamma\in [0,2\alpha-\frac{m}{p}-1)$. Moreover, $V:=(X^{\alpha},\|\cdot\|_{\alpha})$ is a strongly ordered Banach space with the closed-convex cone $X^{\alpha}_{+}$ consisting of all nonnegative functions in $X^{\alpha}$. For every $u_{0}\in V$, (\ref{E1}) admits a (locally) unique regular solution $u(t,x,u_{0})$ in $V$. Set $X\triangleq\{v\in X^{\alpha}:-\kappa< v(x)<\kappa\ for\ all\ x\in\Omega\}$ for $\kappa$ in (N4). By virtue of (N4) and the standard a priori estimates (\cite{H81}), the solution $u$ is bounded in $V$, and hence becomes a globally defined classical solution on $X$.

Let $F$ be the $\tau$-periodic Poincar\'e map of (\ref{E1}) as $F:V\rightarrow V, u(t)\mapsto u(t+\tau)$. Due to Strong Maximum Principle, $F$ generates a discrete-time strongly monotone system on $V$. One further knows that $F$ is of $C^{1}$ (see, e.g. \cite{H81}); and moreover, $F$ is $\omega$-compact in every closed order interval in $X$ since $F$ is a compact map with $F(X)\subseteq X$ (see, e.g. \cite{P,H91}).

Thus, the discrete-time system generated by $F$ satisfies the standing assumptions (H1)-(H3) in Section \ref{sec2}. Based on our Theorem A, we have the following theorem.
\begin{thm}\label{ex1}
 Assume that \emph{(N1)-(N4)} hold for system \emph{(\ref{E1})}. Then the set of initial condition $u_{0}\in X$ corresponding to solution $u(t,x)$ converges to a linearly stable $k\tau$-periodic solution \emph{(}for some integer $k>0$\emph{)}, is prevalent in $X$.
\end{thm}


\vspace{4 ex}
\noindent\textbf{Example 2 (Prevalent dynamics of parabolic equations on convex domains)}
Consider the periodic-parabolic Neumann problem on a smooth convex bounded domain $\Omega\subset\mathbb{R}^{m}$:
\begin{eqnarray}\label{E2}
        \frac{\partial u}{\partial t} -\triangle u &=& f(t,u,\nabla u), \quad\quad\ \rm{in}\ \Omega\times\mathbb{R},\\
        \frac{\partial u}{\partial n} &=& 0, \quad\quad\quad\quad\quad\quad\ \rm{on}\ \partial\Omega\times\mathbb{R},\notag
\end{eqnarray}
 where $f:(t,u,\xi)\mapsto f(t,u,\xi)$ is independent of $x$ and satisfies assumptions (N1)-(N4).\par

Again, set $V:=(X^{\alpha},\|\cdot\|_{\alpha})$ and $X\triangleq\{u\in X^{\alpha}:-\kappa< u(x)<\kappa\ for\ all\ x\in\Omega\}$. Let $F$ be the $\tau$-periodic Poincar\'e map as
$F:V\rightarrow V, u(t)\mapsto u(t+\tau)$.
So $F$ also generates a $C^{1}$-smooth discrete-time dynamical system on $V$ satisfying (H1)-(H3). We call a solution $u(t,x)$ is \emph{spatially homogeneous} if $u(t,x)$ is independent of spatial variable $x$.

\begin{thm}\label{ex2}
Assume that \emph{(N1)-(N4)} hold for system \emph{(\ref{E2})}. Then the set of initial condition $u_{0}\in X$, whose solution $u(t,x)$ converges to a spatially homogeneous $\tau$-periodic solution, is prevalent in $X$.
\end{thm}
\begin{proof}
Since $F$ satisfies (H1)-(H3) in Section \ref{sec2}, it follow from Theorem \ref{ex1} that the set of initial conditions whose solutions are asymptotic to linearly stable periodic solutions, is prevalent in $X$. Recall that Hess \cite{PH} has shown that $u$ is spatially homogeneous if it is a linearly stable periodic solution of (\ref{E2}). And Pol\'a\v{c}ik and Tere\v{s}\v{c}\'ak \cite{PT93} further showed that such $u$ is a fixed point of $F$, which means $u$ is $\tau$-periodic. Thus, we obtain that the set of initial condition $u_{0}$ of (\ref{E2}), whose solution $u(x,t)$ converges toward a spatially homogeneous $\tau$-periodic solution, is prevalent in $X$.
\end{proof}

\vspace{4 ex}


\noindent\textbf{Example 3 (Prevalent dynamics of parabolic equations on $SO(N)$-invariant domains)} Let $G$ be a compact connected subgroup of the special orthogonal group $SO(N)$. The domain $\Omega$ is said to be $G$-invariant if $a\cdot x\in\Omega$ for all $x\in\Omega, a\in G$. Now we consider the equation (\ref{E1}) with the open bounded $G$-invariant domain $\Omega\subset\mathbb{R}^{m}$ with a $C^{2+\mu}$-boundary $\partial\Omega$, for some $\mu\in (0,1)$,
\begin{eqnarray}\label{group}
        \frac{\partial u}{\partial t} -\triangle u &=& f(t,x,u,\nabla u), \quad\quad\quad\ \ x\in \Omega,\ t>0,\notag\\
       \mathcal{B}u &=& 0, \quad\quad\quad\quad\quad\quad\quad\quad\ x\in \partial\Omega,\ t>0,\\
       u(0,x) &=& u_{0}(x),\notag
\end{eqnarray}
where $\mathcal{B}u=u$ or $\mathcal{B}u=\frac{\partial u}{\partial v}$.\par
We assume the nonlinearity $f$ satisfies (N1)-(N4). Moreover, we assume that $f$ satisfies:\par

\vspace{2 ex}

\textbf{(N5)} $f(t,a\cdot x,u,a\cdot \xi)=f(t,x,u,\xi)$ for all $a\in G$ and $(t,x,u,\xi)\in
\mathbb{R}_{+}^{1}\times\bar{\Omega}\times\mathbb{R}\times\mathbb{R}^{m}$.\par

\vspace{2 ex}

\noindent Take $X=C(\bar{\Omega},\mathbb{R})$. The action of $G$ on $X$ is defined by $a\cdot u=u\circ a$, i.e., $\gamma(a,u)(x)=u(a\cdot x)$ for all $a\in G$, $u\in X$ and $x\in \bar{\Omega}$. Let $F$ be the $\tau$-periodic Poincar\'e map as $F:X\rightarrow X, u(t)\mapsto u(t+\tau)$, which is $C^{1}$. Then $\gamma$ commutes with $F$ (see Tak\'a\v{c} \cite{P920}). We call a solution $u(t,x)$ is \emph{$G$-symmetric} if $u(t,a\cdot x)=u(t,x)$ for any $a\in G$ and $t\geq0$. According to the Theorem B, the prevalent solutions of equation (\ref{group}) that converge to symmetric functions, i.e., we can obtain that 

\begin{thm}\label{ex33}
Assume that \emph{(N1)-(N5)} hold for system \emph{(\ref{group})}. Then the set of initial condition $u_{0}\in X$ corresponding to solution $u(t,x)$ converges to a $G$-symmetric linearly stable $k\tau$-periodic solution $p(x,t)$ \emph{(}for some integer $k>0$\emph{)}, is also prevalent in $X$.
\end{thm}



\vspace{4 ex}
\noindent\textbf{Example 4 (Prevalent dynamics of reaction-diffusion equations on a ball)} Consider the prevalent dynamics of the following Dirichlet initial-boundary value problem on an $N$-dimensional ball $B=\{x\in \mathbb{R}^{N}:\|x\|<1\}$:
\begin{eqnarray}\label{E3.2}
        \frac{\partial u}{\partial t} -\triangle u &=&  f(t,u),\quad\quad  t>0,\ x\in B,\notag\\
        u(x,t) &=& 0, \quad\quad\quad\quad\quad\quad\quad t>0,\ x\in \partial B,\\
        u(0,x) &=& u_{0}(x), \quad\quad\quad\quad\quad x\in \bar{B}.\notag
\end{eqnarray}
Again, we assume that (N1)-(N4) hold. Let $V$, $X$ and $F$ be defined as in Example \ref{E1}. Take the group $G=SO(N)$. Then, we have the following theorem.
\begin{thm}\label{sym}
There exists a prevalent set $S\subset X$ such that for any solution $u(x,t)$ of equation \emph{(\ref{E3.2})} with initial value $u_{0}\in S$, there exists a linearly stable $\tau$-periodic solution $p(x,t)$ such that\par
\emph{(i)} $p(\cdot,t)$ is radially symmetric, i.e., $p(x,t)=p(|x|,t)$;\par
\emph{(ii)} the solution $\|u(\cdot,t)-p(\cdot,t)\|_{X}\rightarrow 0$ as $t\rightarrow\infty$.
\end{thm}

\begin{proof}
Let $V_{\rm{rad}}$ be the space of all radially symmetric functions in $V$ and let $X_{\rm{rad}}=X\cap V_{\rm{rad}}$. We indentify a function $u\in X_{\rm{rad}}$ with its radial version $U\in X^{\alpha}\subset C^{1}[0,1]$ with $$U(\|x\|,t)=u(x,t),\quad U'(0)=U(1)=0.$$
Define by $F_{\rm{rad}}$ be the Poincar\'e-map associated with the following 1-D equation on $X_{\rm{rad}}$:
\begin{eqnarray}\label{E3.3}
        U_{t} &=& U_{rr} +\frac{N-1}{r}U_{r} + F(t,U),\quad\quad\quad  t>0,\ 0<r<1,\notag\\
        U_{r}(0,t)&=&U(1,t)=0, \quad\quad\quad\quad\quad\quad\quad\quad\quad t>0,\\
        U(r,0) &=& U_{0}(r), \quad\quad\quad\quad\quad\quad\ \quad\quad\quad\quad\quad 0\leq r\leq 1.\notag
\end{eqnarray}
where $r=\|x\|$ and $F(t,U)=f(t,u)$.

By virtue of Theorem B, one can find a prevalent set $S\subset X$ such that, for any $u(x,t)$ with initial value $u_{0}\in S$, $u$ converges in $X$ to a linearly stable periodic solution $p(r,t)$ of the (\ref{E3.3}) in $X_{\rm{rad}}$. It follows from Chen and Pol\'a\v{c}ik \cite[Theorem 3.5]{CP96} that any chain recurrent set of $F_{\rm{rad}}$ consists of fixed points of $F_{\rm{rad}}$. Then $p\in X_{\rm{rad}}$ and $p(\cdot,t)=p(\cdot,t+\tau)$. Thus, we obtain that $\|u(\cdot,t)-p(\cdot,t)\|_{X}\rightarrow 0$ as $t\rightarrow\infty$, which completed the proof.
\end{proof}

\begin{rmk}
Chen and Pol\'a\v{c}ik \cite{CP96} have proved the asymptotic $\tau$-periodicity of \emph{positive} solutions. Our Theorem \ref{sym} showed that prevalence of the asymptotic $\tau$-periodicity for solutions of (\ref{E3.2}).
\end{rmk}


\begin{thebibliography}{10}


\bibitem{CP96}
X. Chen, and P. Pol\'a\v{c}ik, \rm{Asymptotic periodicity of positive solutions of reaction diffusion equations on a ball},
J. Reine Angew. Math. \textbf{472}(1996), 17-51.

\bibitem{C72}
J. Christensen, \rm{On sets of haar measure zero in abelian Polish groups},
Isr. J. Math. \textbf{13}(1972), 255-260.

\bibitem{I02}
I. Chueshov, \rm{\emph{Monotone random systems-theory and applications}},
Lecture Notes in Mathematics, Vol. 1779, Springer-Verlag, Berlin, 2002.

\bibitem{DP91}
E. Dancer and P. Hess, \rm{Stability of fixed-points for order-preserving discrete-time dynamic-systems},
J. Reine Angew. Math. \textbf{419}(1991), 125-139.


\bibitem{EN20}
M. Elekes and D. Nagy, \rm{Haar null and haar meager sets: a survey and new results},
Bull. London Math. Soc. \textbf{52}(2020), 561-619.


\bibitem{GMH08}
G. Enciso, M. W. Hirsch and H. Smith, \rm{Prevalent behavior of strongly order preserving semiflows},
J. Dynam. Diff. Eqns. \textbf{20}(2008), 115-132.

\bibitem{FWW19}
L. Feng, Y. Wang and J. Wu, \rm{Generic behavior of flows strongly monotone with respect to high-rank cones},
J. Diff. Eqns. \textbf{275}(2021), 858-881.

\bibitem{FWW17}
L. Feng, Y. Wang and J. Wu, \rm{Semiflows ``Monotone with Respect to High-Rank Cones" on a Banach space},
SIAM J. Math. Anal. \textbf{49}(2017), 142-161.

\bibitem{H81}
D. Henry, \rm{\emph{Geometric theory of semilinear parabolic equations}},
Springer, New York, 1981.


\bibitem{PH}
P. Hess, \rm{Spatial homogeneity of stable solutions of some periodic-parabolic problems with neumann boundary conditions},
J. Diff. Eqns. \textbf{68}(1987), 320-331.

\bibitem{H91}
P. Hess, \rm{\emph{Periodic parabolic boundary-value problems and positivity}},
Pitman Reasearch Notes in Mathematics, Vol. 247, Longman Scientific and Technical, New York, 1991.

\bibitem{H84}
M. W. Hirsch, \rm{The dynamical systems approach to differential equations},
Bull. Amer. Math. Soc. \textbf{11}(1984), 1-64.

\bibitem{H88}
M. W. Hirsch, \rm{Stability and convergence in strongly monotone dynamical systems},
J. Reine Angew. Math. \textbf{383}(1988), 1-53.

\bibitem{H82}
M. W. Hirsch, \rm{Systems of differential equations which are competitive or cooperative I: limit sets},
SIAM J. Appl. Math. \textbf{13}(1982), 167-179.

\bibitem{H85}
M. W. Hirsch, \rm{Systems of differential equations which are competitive or cooperative II: convergence almost everywhere},
SIAM J. Math. Anal. \textbf{16}(1985), 423-439.






\bibitem{MH05}
M. W. Hirsch and H. Smith, \rm{\emph{Monotone dynamical systems}},
Handbook of Differential Equations: Ordinary Differential Equations, Vol. 2, Elsevier, Amsterdam 2005.

\bibitem{HSY93}
B. Hunt, T. Sauer and J. Yorke, \rm{Prevalence: a transaltion-invariant `almost every' on infinite-dimensional spaces},
Bull. Amer. Math. Soc. \text{27}(1992), 217-238.

\bibitem{BV10}
B. Hunt and V. Kaloshin, \rm{\emph{Prevalence}}, Handbook of Dynamical Systems. Vol. 3, Elsevier, 2010.

\bibitem{JX05}
J. Jiang and X. Zhao, \rm{Convergence in monotone and uniformly stable skew-product semiflows with applications},
J. Reine Angew. Math. \textbf{589}(2005), 21-55.

\bibitem{K97}
V. Kaloshin, \rm{Some prevalent properties of smooth dynamical systems},
Proc. Steklov Inst. Math. \textbf{213}(1997), 123-151.





\bibitem{M79}
H. Matano, \rm{Asymptotic behavior and stability of solutions of semilinear diffusion equations},
Publ. RIMS Kyoto Univ. \textbf{15}(1979), 401-454.

\bibitem{M86}
H. Matano, \rm{Strongly order-preserving local semi-dynamical systems-theory and applications},
in Semigroups, Theory and Applications, Vol. 1, H. Brezis, M. G. Crandall, and F. Kappel, eds., Res. Notes in Math., \textbf{141}(1986), Longman Scientific and Technical, London, 178-185.


\bibitem{M94}
J. Mierczy\'nski, \rm{P-arcs in strongly monotone discrete-time dynamical systems},
Differential and Integral Equations, \textbf{7}(1994), 1473-1494.

\bibitem{M89}
J. Mierczy\'nski and P. Pol\'a\v{c}ik, \rm{Group actions on strongly monotone dynamical systems},
Math. Ann. \textbf{283}(1989), 1-11.


\bibitem{P89}
P. Pol\'a\v{c}ik, \rm{Convergence in smooth strongly monotone flows defined by semilinear parabolic equations},
J. Diff. Eqns. \textbf{79}(1989), 89-110.

\bibitem{P90}
P. Pol\'a\v{c}ik, \rm{Generic properties of strongly monotone semiflows defined by ordinary and parabolic differential equations},
Qualitative theory of differential equations (Szeged 1988) 519-530, Collop. Math. Soc. J\'anos Bolyai, 53, North-Holland, Amsterdam, 1990.

\bibitem{P}
P. Pol\'a\v{c}ik, \rm{\emph{Parabolic equations: asymptotic behavior and dynamics on invariant manifolds}},
Handbook on Dynamical Systems, Vol. 2, Elsevier, Amsterdam, 2002, 835-883.





\bibitem{PI91}
P. Pol\'a\v{c}ik and I. Tere\v{s}\v{c}\'{a}k, \rm{Convergence to cycles as a typical asymptotic behavior in smooth strongly monotone discrete-time dynamical systems},
Arch. Ration. Mech. Anal. \textbf{116}(1991), 339-360.

\bibitem{PT93}
P. Pol\'a\v{c}ik and I. Tere\v{s}\v{c}\'ak, \rm{Exponential separation and invariant bundles for maps in ordered Banach spaces with applications to parabolic equations},
J. Dynam. Diff. Eqns, \textbf{5}(1993), 279-303.


\bibitem{WY98}
W. Shen and Y. Yi, \rm{Almost automorphic and almost periodic dynamics in skew-product semiflows},
Mem. Amer. Math. Soc. \textbf{136}(1998), No. 647.

\bibitem{H}
H. Smith, \rm{\emph{Monotone Dynamical Systems: an introduction to the theory of competitive and cooperative systems}},
Math. Surveys and Monographs, Vol. 41, Amer. Math. Soc., Providence, Rhode Island, 1995.

\bibitem{H17}
H. Smith, \rm{Monotone Dynamical Systems: Reflections on new advanves and applications},
Discrete Contin. Dyn. Syst. \textbf{37}(2017), 485-504.

\bibitem{ST91}
H. Smith and H. Thieme, \rm{Convergence for strongly order-preserving semiflows},
SIAM J. Math. Anal. \textbf{22}(1991), 1081-1101.

\bibitem{P920}
P. Tak\'a\v{c}, \rm{Asymptotic behavior of strongly monotone time-periodic dynamical processes with symmetry},
J. Diff. Eqns. \textbf{100}(1992), 355-378.

\bibitem{P92}
P. Tak\'a\v{c}, \rm{Domains of attraction of generic $\omega$-limit sets for srongly monotone discrete-time semigroups},
J. Reine Angew. Math. \textbf{423}(1992), 101-173.

\bibitem{T94}
I. Tere\v{s}\v{c}\'{a}k, \rm{Dynamics of ${C}^{1}$ smooth strongly monotone discrete-time dynamical systems},
preprint, Comenius University, Bratislava, 1994.

\bibitem{W09}
Y. Wang, \rm{Asymptotic symmetry in strongly monotone skew-product semiflows with applications},
Nonlinearity \textbf{22}(2009), 765-782.

\bibitem{WY20}
Y. Wang and J. Yao, \rm{Dynamics alternatives and generic convergence for $C^{1}$-smooth strongly monotone discrete dynamical systems},
J. Diff. Eqns. \textbf{269}(2020), 9804-9818.
\end{thebibliography}
\end{document}